\pdfoutput=1
\RequirePackage{ifpdf}
\ifpdf 
\documentclass[pdftex]{sigma}
\else
\documentclass{sigma}
\fi

\numberwithin{equation}{section}
\newtheorem{Theorem}{Theorem}[section]
\newtheorem{Lemma}[Theorem]{Lemma}
\newtheorem{Proposition}[Theorem]{Proposition}
 { \theoremstyle{definition}
\newtheorem{Definition}[Theorem]{Definition}
 }
\newcommand{\half}{\ensuremath{\frac{1}{2}}}
\newcommand{\N}{{\mathbb N}}
\newcommand{\R}{{\mathbb R}}
\newcommand{\cS}{{\mathcal S}}
\newcommand{\va}{{\boldsymbol{a}}}
\newcommand{\vb}{{\boldsymbol{b}}}
\newcommand{\vy}{{\boldsymbol{y}}}
\newcommand{\vzero}{{\boldsymbol{0}}}
\newcommand{\Sdmone}{{{\mathbb S}^{d-1}}}
\newcommand{\bbS}{{\mathbb S}}
\newcommand{\bbN}{{\mathbb N}}
\newcommand{\HypFunTO}[4]{{
 {_2F_1}\left[ \left.
 \begin{matrix}
 #1 , #2\\
 #3 \\
 \end{matrix} \right| #4 \right] }} 
\def\CD{{\mathcal D}}
\def\CF{{\mathcal F}}
\def\CH{{\mathcal H}}
\def\CX{{\mathcal X}}
\def\NN{{\mathbb N}}
\def\RR{{\mathbb R}}
\def\SS{{\mathbb S}}
\def\Li{\operatorname{Li}}

\begin{document}
\allowdisplaybreaks

\newcommand{\arXivNumber}{1801.01313}

\renewcommand{\PaperNumber}{083}

\FirstPageHeading

\ShortArticleName{Thinplate Splines on the Sphere}

\ArticleName{Thinplate Splines on the Sphere}

\Author{Rick K.~BEATSON~$^\dag$ and Wolfgang ZU~CASTELL~$^{\ddag\S}$}

\AuthorNameForHeading{R.K.~Beatson and W.~zu~Castell}

\Address{$^\dag$~School of Mathematics and Statistics, University of Canterbury,\\
\hphantom{$^\dag$}~Private Bag 4800, Christchurch, New Zealand}
\EmailD{\href{mailto:r.beatson@math.canterbury.ac.nz}{r.beatson@math.canterbury.ac.nz}}

\Address{$^\ddag$~Scientific Computing Research Unit, Helmholtz Zentrum M\"{u}nchen,\\
\hphantom{$^\ddag$}~Ingolst\"{a}dter Landstra{\ss}e~1, 85764 Neuherberg, Germany}
\EmailD{\href{mailto:castell@helmholtz-muenchen.de}{castell@helmholtz-muenchen.de}}

\Address{$^\S$~Department of Mathematics, Technische Universit\"{a}t M\"{u}nchen, Germany}

\ArticleDates{Received January 08, 2018, in final form July 30, 2018; Published online August 12, 2018}

\Abstract{In this paper we give explicit closed forms for the semi-reproducing kernels associated with thinplate spline interpolation on the sphere. Polyharmonic or thinplate splines for ${\mathbb R}^d$ were introduced by Duchon and have become a widely used tool in myriad applications. The analogues for ${\mathbb S}^{d-1}$ are the thin plate splines for the sphere. The topic was first discussed by Wahba in the early 1980's, for the ${\mathbb S}^2$ case. Wahba presented the associated semi-reproducing kernels as infinite series. These semi-reproducing kernels play a central role in expressions for the solution of the associated spline interpolation and smoothing problems. The main aims of the current paper are to give a recurrence for the semi-reproducing kernels, and also to use the recurrence to obtain explicit closed form expressions for many of these kernels. The closed form expressions will in many cases be significantly faster to evaluate than the series expansions. This will enhance the practicality of using these thinplate splines for the sphere in computations.}


\Classification{42A82; 33C45; 42C10; 62M30}

\section{Introduction}

In this paper we give explicit closed forms for the semi-reproducing kernels associated with thinplate spline interpolation on the sphere. Polyharmonic or thinplate splines for $\RR^d$ were introduced by Duchon in his classic papers~\cite{Duchon_76,Duchon_77} and have become a widely used tool in myriad applications. The analogues for $\Sdmone \subset \R^d$ are the thinplate splines for the sphere. The topic was first discussed by Wahba~\cite{Wahba81,Wahba82} in the early 1980's, for the $\bbS^2$ case. Wahba presented the associated semi-reproducing kernels as infinite series. These semi-reproducing kernels play a~central role in expressions for the solution of the associated spline interpolation and smoothing problems.

The main aims of the current paper are to give a recurrence for these semi-reproducing kernels, and also to use the recurrence to obtain explicit closed form expressions. Here we are building on previous work of Martinez-Morales~\cite{Ma05}. Unfortunately, there are errors in the theory presented in~\cite{Ma05} and consequently many of the expressions given there for the kernels are incorrect. The closed form expressions given here will usually be significantly faster to evaluate than the series expansions. This will enhance the practicality of using the thinplate splines for the sphere in computations.

The paper is laid out as follows. Section~\ref{sec:reproducing_kernels} discusses the central role played by semi-re\-pro\-ducing kernels in the solution of both interpolation and penalized least squares fitting problems. Section~\ref{sec:Thinplate_splines_on_the_sphere} develops semi-reproducing kernels associated with the thinplate splines on the sphere, that is, semi-reproducing kernels associated with minimum energy interpolation and penalized least squares fitting problems with a particular choice of energy. The energy chosen being that naturally associated with iterated Laplace--Beltrami operators. These semi-reproducing kernels are given in this section as infinite series. Section~\ref{sec:Thinplate_splines_on_the_sphere} also recalls known results concerning Fourier--Gegenbauer expansions that will be needed later. Section~\ref{sec:the_operators_T_and_T*} motivates the construction of an operator $T$ and its adjoint $T^*$. It also presents some fundamental properties of these operators. These operators were initially developed in Martinez-Morales~\cite{Ma05}. Section~\ref{sec:recurrence} gives a~recurrence for the various thinplate spline kernels $K_{d,m}(x,y)$, where $d$ indicates the dimension and $m$ is the power of the associated differential operator. More precisely, it gives a recurrence for the related functions $k_{d,m}\colon [-1,1]\rightarrow \R$ where $k_{d,m}(\langle x, y \rangle)= K_{d,m}(x,y)$. Sections~\ref{sec:explicit_forms} and \ref{sec:m_equal_1} give short closed form expressions for many of the functions $k_{d,m}$.

For ease of access to the relevant background we will base our notation on that used in Dai and Xu~\cite{DaiXu2013}. Occasionally, when the value for the dimension $d$ is particularly important, we will supplement the symbol they use with a $d$.

\section{Reproducing kernels and approximation} \label{sec:reproducing_kernels}

Let us begin with summarising the main ideas of reproducing kernels in indefinite inner product spaces and their role in both interpolation and penalized least squares fitting problems. These results show the central role of reproducing kernels in the solution of these problems. We focus on the specific case of reproducing kernels\footnote{Note that reproducing kernels for semi-Hilbert spaces also appear as \emph{semi-kernels} \cite{BerlinetAgnan,Laurent1991} or \emph{increment reproducing kernels} \cite{MosamamKent2010}.} for semi-Hilbert spaces, or simply semi-reproducing kernels, as this is the appropriate framework for thinplate spline approximation. For further definitions and basic properties of semi-reproducing kernels we refer to \cite{BerlinetAgnan, BezaevValisenko, CheneyLight_2000, MosamamKent2010}. Our treatment of the relevant interpolation and penalized least squares problems is based on that of Strauss~\cite{Strauss_2002}.

Let $\CD$ be a subset of $\RR^d$. Consider approximation from a vector space over the reals $\CH \subset C(\CD)$. Assume the space $\CH$ is endowed with a semi-inner product $( \cdot, \cdot)$, that is, the inner product is lacking definiteness. Thus, there are non-zero vectors $f\in\CH$ with $(f,f)=0$. Further assume that the kernel $\CH_0$ of the semi-inner product $(\cdot,\cdot)$ is finite-dimensional, i.e., $\dim\CH_0=m<\infty$, and $(f,f)=0$ if and only if $f\in \CH_0$.

A standard approach to deal with semi-inner products is to supplement the semi-inner product with an inner product on $\CH_0$ thereby obtaining a~definite inner product on the space $\CH$ (see~\cite{Bognar1974}). Towards this aim, we need to decompose the space $\CH$ into a direct sum $\CH_0\oplus\CH_1$, such that the given semi-inner product provides a definite inner product on the subspace $\CH_1$. Given $m$ linearly independent functionals spanning the dual of $\CH_0$, we define $\CH_1$ as the space of functions in $\CH$ which are mapped onto zero by all these functionals.

Let us recall this approach using point evaluations. Nevertheless, it is important to note that there are many choices for such a set of functionals. Clearly, the decomposition obtained for the space $\CH$ depends upon the choice made for the $m$ functionals.

\begin{Definition} A set of distinct points $\CX $ is said to be unisolvent for $\CH_0$ if the only function in~$\CH_0$ which is zero at all points of $\CX$ is the zero function.
\end{Definition}
Given a unisolvent set $\CX=\{z_1, \ldots,x_m\}$ for $\CH_0$, where $m=\dim (\CH_0)$, the set of point evaluation functionals $\{\delta_{x}\colon x\in\CX \}$ is linearly independent on~$\CH_0$. Hence, we can find a~Lagrange basis $u_1,\dots, u_m$ of $\CH_0$ with respect to $\CX$, i.e., $u_i(x_j)=\delta_{ij}$, $1\leq i,j \leq m$. Then
\begin{gather*}
[f,g]_0 = \sum_{j=1}^m f(x_j)g(x_j), \qquad f,g\in\CH_0,
\end{gather*}
defines an inner product on $\CH_0$. The basis $u_1,\dots, u_m$ is orthogonal with respect to the inner product
$[\cdot,\cdot]_0$. Furthermore, the mapping
\begin{gather*}
P_0\colon \ \CH \to \CH_0, \qquad f\mapsto f_0=\sum_{j=1}^m f(x_j) u_j
\end{gather*}
is a projection of $\CH$ onto $\CH_0$. The definition of inner product for the full space $\CH$ which follows will make $P_0$ the orthogonal projection onto $\CH_0$. The subspace $\CH_1$ can then be defined via the projector $P_1=I-P_0$, i.e.,
\begin{gather*}
\CH_1 = \{ f\in\CH \colon f(x)=0 \ \text{for all} \ x\in\CX\}.
\end{gather*}
Since $\CH_0$ is the kernel of the semi-inner product $(\cdot,\cdot)$, i.e., $(f,f)=0$ iff $f\in\CH_0$, the semi-inner product is definite on $\CH_1$ via construction. Therefore,
\begin{gather*}
[f,g] = [P_0f, P_0 g]_0 + (P_1f, P_1g), \qquad f,g \in \CH,
\end{gather*}
defines a definite inner product on $\CH$ (see \cite{Bognar1974} for further details).

We are interested in Hilbert spaces carrying the special property of being reproducing kernel spaces. There is a one-to-one correspondence between reproducing kernel Hilbert spaces and positive definite kernels. A~similar relation holds true for semi-reproducing kernel Hilbert spaces.

\begin{Definition}Given $n\in\NN$, a pair $(\CX,\va)$ with $\CX=\{x_1,\dots,x_n\}$ a set of distinct points from~$\CD$ and $\va=(a_1,\dots,a_n)^{\rm T}\in\RR^n$ is called an $\CH_0$-increment if
\begin{gather*}
\sum_{j=1}^n a_j f(x_j) = 0 \qquad\text{for all} \ f\in\CH_0.
\end{gather*}
The set of all $\CH_0$-increments is denoted by $\CH_0^\perp$.
\end{Definition}
Note that an $\CH_0$-increment can naturally be identified with a linear functional
\begin{gather*}
\lambda_{\CX,\va}(f) = \sum_{j=1}^n a_j \delta_{x_j}(f), \qquad f\in \CH,
\end{gather*}
vanishing on $\CH_0$.

The structure of semi-reproducing kernels already shows that we can expect the reproducing kernel for a semi-Hilbert space to be positive definite only on a suitable subspace.

\begin{Definition} Let $\CD \subset \RR^d$, $\CH_0$ be a finite dimensional subset of $C(\CD)$ and $K\colon \CD \times \CD \rightarrow \RR$ be a symmetric function. $K$ is called conditionally positive definite with respect to $\CH_0$ if for all $n\in\NN$
\begin{gather}\label{eq:cpd_of_P_wrt_U}
\sum_{i=1}^n\sum_{j=1}^n a_i a_j K(x_i,x_j) \geq 0,
\end{gather}
for all $\CH_0$-increments $(\CX,\va)$. $K$ is said to be strictly conditionally positive definite with respect to $\CH_0$ if the inequality in~\eqref{eq:cpd_of_P_wrt_U} is strict whenever in addition $\va$ is nonzero.
\end{Definition}
There is a correspondence between conditionally positive definite kernels and semi-repro\-du\-cing spaces (see \cite{Atteia,Meinguet,MosamamKent2010}). The associated reproducing property can again be stated in terms of $\CH_0$-increments.

\begin{Definition}\label{def:semi-rep-kernel}Let $\CH\subset C(\CD)$ be a semi-Hilbert space, i.e., $\CH$ is a semi-inner product space with semi-inner product $(\cdot, \cdot)$ the kernel $\CH_0$ of which is finite-dimensional, and $\CH$ is complete with respect to the induced semi-norm. A symmetric kernel $K\colon \CD \times \CD \rightarrow \RR$ is called a semi-reproducing kernel for $\CH$ if $(\cdot, \cdot)$ reproduces $\CH_0$-increments, i.e., for all $\lambda_{\CX,\va}$ annihilating $\CH_0$ the following two properties hold
 \begin{gather}
 \sum_{j=1}^n a_j K(\cdot, x_j) \in \CH, \label{eq:first_prop_semi_repro_kernel}
 \end{gather}
 and
 \begin{gather}
 \left( f, \sum_{i=1}^n a_i K(\cdot, x_i) \right) = \sum_{j=1}^n a_j f(x_j) \qquad \text{for all} \ f \in \CH. \label{eq:second_prop_semi_repro_kernel}
 \end{gather}
 \end{Definition}
Given $m=\dim\CH_0$ linearly independent functionals $\lambda_1, \dots, \lambda_m$ on $\CH_0$ and the corresponding Lagrange basis
\begin{gather*}
u_1,\dots, u_m \qquad\mbox{ such that }\quad \lambda_j(u_i)=\delta_{ji}, \quad 1\leq i,j\leq m,
\end{gather*}
the kernel
\begin{gather*}
K_0(x,y) = \sum_{j=1}^m u_j(x)u_j(y), \qquad x,y\in \CD,
\end{gather*}
obviously provides a reproducing kernel for $\CH_0$ with respect to the inner product
\begin{gather*}
[f,g]_0 = \sum_{j=1}^m \lambda_j(f)\lambda_j(g), \qquad f,g\in \CH_0.
\end{gather*}
Using the orthogonal projection
\begin{gather*}
P_0f = \sum_{j=1}^m \lambda_j(f) u_j, \qquad f\in\CH,
\end{gather*}
we can again define $\CH_1=(I-P_0)\CH$. By construction,
\begin{gather*}
\CH_1 = \{ f\in\CH\colon \lambda_j(f)=0,\, 1\leq j\leq m \}.
\end{gather*}

Furthermore, if $K$ is a semi-reproducing kernel of $\CH$ the kernel
\begin{gather*}
K_1(x,y) = K(x,y) - \sum_{j=1}^m u_j(y)\lambda_j\big(K(x,\cdot)\big) - \sum_{j=1}^m u_j(x) \lambda_j(K\big(\cdot,y)\big) \\
\hphantom{K_1(x,y) =}{} + \sum_{i=1}^m\sum_{j=1}^m u_i(x)u_j(y) \lambda_i^1\lambda_j^2\big(K(\cdot,\cdot)\big),
\end{gather*}
is the reproducing kernel of $\CH_1$. Here, the superindex in the last term indicates the functional operating on the first and second variable, respectively, Thus, the space $\CH$ is a reproducing kernel Hilbert space itself with reproducing kernel
\begin{gather*}
K_\CH(x,y) = K_1(x,y) + K_0(x,y), \qquad x,y\in \CD.
\end{gather*}
See \cite{CheneyLight_2000, MosamamKent2010} for details. Note that if $K$ is given such that
$P_0K(\cdot,x)=0$ for all $x\in\CD$, then the projection in the above expression vanishes, i.e. $K_1=K$.

The framework of reproducing kernel spaces allows us to consider regularized interpolation problems in broad mathematical generalities (see \cite[Chapter~2.1]{BerlinetAgnan}). The rather beautiful result of Strauss~\cite{Strauss_2002} concerning mixed interpolation and regularized least squares problems provides a special case of \cite[Theorem~59]{BerlinetAgnan}.

The following notation will be used. For a set of points $X=\{x_1,\dots,x_n\}\subset\CD$ and a~func\-tion~$f$ on~$\CD$ we write $f_X$ for the vector $(f(x_1),\dots,f(x_n))^{\rm T}$, $K_X$ for the $n\times n$ matrix with $ij$-entry $K(x_i,x_j)$, and~$C_X$ for the $m\times n$ matrix $\big(u_i(x_j)\big)$, where $u_1,\dots,u_m$ is a basis of $\CH_0$. Furthermore, $W^\dagger$~denotes the pseudo-inverse of the matrix $W$ appearing in the assumptions of the following theorem. Under the assumptions of the theorem $W^\dagger = \left[\begin{smallmatrix} R^{-1} & O \\ O & O \end{smallmatrix}\right]$.

\begin{Theorem} \label{thm:Interpolation_and_smoothing} Let $\CH$ and $(\cdot, \cdot)$ be as in Definition~{\rm \ref{def:semi-rep-kernel}}, and let $K$ be a semi-reproducing kernel for $(\CH, (\cdot, \cdot))$ with respect to~$\CH_0$. Further suppose that $K$ is strictly conditionally positive definite with respect to~$\CH_0$. Let $\mu > 0$, $n \geq m=\text{dim } \CH_0$, $ 0\leq p \leq n$, and $W$ be an $n \times n$ matrix of the form
\begin{gather*}
W=\begin{bmatrix} R & O\\
 O & O \end{bmatrix},
\end{gather*}
where $R$ is $p \times p$ and symmetric positive definite. Then given any set $X=\{x_1, \ldots, x_n\}$ of $n$ distinct points in $\CD$ which is unisolvent for $\CH_0$, and $n$ corresponding values $y_i\in\RR$, there is a~unique member of the space~$\CH$ minimizing the quadratic functional
 \begin{gather*}
 (f_{X} - \vy)^{\rm T} W^\dagger (f_{X}-\vy) +\mu (f,f),
 \end{gather*}
over those functions in $\CH$ which satisfy the interpolation conditions
 \begin{gather*}
 f(x_i) = y_i, \qquad p+1 \leq i \leq n.
 \end{gather*}
 This function can be written in the form
 \begin{gather*}
 s = \sum_{i=1}^n a_i K(\cdot, x_i) + \sum_{i=1}^m b_i u_i,
 \end{gather*}
where the coefficients $\va=(a_1,\dots,a_n)^{\rm T}$ and $\vb=(b_1,\dots,b_m)^{\rm T}$ are the solution of the system
 \begin{gather*}
 (K_{X} + \mu W ) \va +C_{X} ^{\rm T} \vb = \vy , \qquad C_{X} \va = \vzero .
 \end{gather*}
\end{Theorem}
Note that the statement reduces to the well-known result concerning the solution of the smoothest interpolation problem when $p=0$, and to a known expression for the smoothing spline when $p=n$.

\section[Series representations for thinplate spline kernels on the sphere]{Series representations for thinplate spline kernels\\ on the sphere}\label{sec:Thinplate_splines_on_the_sphere}

For functions on $\RR^d$ interpolating and smoothing with thinplate/polyharmonic splines associated with the energy
\begin{gather*}
E_\kappa(f)= \int_{\RR^d} \sum_{i_1=1,i_2=1,\ldots, i_\kappa=1}^d\left( \frac{\partial}{\partial x_{i_1}} \frac{\partial}{\partial x_{i_2}} \cdots \frac{\partial}{\partial x_{i_\kappa}} f(x)\right)^2 {\rm d}x,
\end{gather*}
are very successful approximation methods. For sufficiently smooth functions $f$, decaying sufficiently fast at infinity, integration by parts gives
\begin{gather*}
E_\kappa(f)= (-1)^\kappa \int_{\RR^d} f(x) (\triangle^\kappa f)(x) {\rm d}x,
\end{gather*}
where $\triangle$ is the Laplacian. Analogues for the sphere come from considering instead of the Laplacian the Laplace--Beltrami operator $ \triangle^\star$, and working on the ``Fourier'' side since the spherical harmonics $\{Y_{nj}\}$ are both a complete orthonormal system for $L^2\big(\Sdmone\big)$ and also eigenfunctions of the Laplace--Beltrami operator. More precisely,
\begin{gather*}
\triangle^\star Y_{j}^n = -n(n+d-2) Y_{j}^n, \qquad j=1,\ldots , N_{d,n}, \qquad n\in\NN_0.
\end{gather*}
For the explicit value of the dimension $N_{d,n}$ see (\ref{eq:Ndn}), below.

This section will consider corresponding spaces of functions on the sphere and the relevant semi-reproducing kernels $K_{d,m}$. The central role played by these semi-reproducing kernels in minimal energy interpolation and regularised least squares fitting is clear from the discussion in Section~\ref{sec:reproducing_kernels}, and in particular Theorem~\ref{thm:Interpolation_and_smoothing}. This topic was first considered by Wahba~\cite{Wahba81} in the $\SS^2$ case. {See also Wahba's monograph~\cite{Wabha90}.} The reader can find valuable additional relevant material in Freeden, Gervens and Schreiner~\cite[Chapter~5]{Fr98}, {Levesley, Light, Ragozin and Sun~\cite{Lev98}}, and Cheney~and Light~\cite[Chapter~32]{CheneyLight_2000}. The material in these references differs somewhat from what appears here. Usually this is due to a treatment based on reproducing kernels rather than semi-reproducing kernels, or to a different choice of the energy.  Gneiting~\cite{Gneiting2013} gives an excellent survey of recent work concerning kernels for the sphere.

Let $\CH^d_n$ denote the space of real harmonic polynomials homogeneous of degree $n$ on~$\R^d$. The spherical harmonics are the restrictions of these to the sphere $\Sdmone$. In a slight abuse of notation the space of spherical harmonics of degree~$n$ on~$\Sdmone$ is also written~$\CH^d_n$. The dimension of the space is
\begin{gather} \label{eq:Ndn}
N_{d,n} = \dim \CH^d_n = \binom{n+d-1}{n} -\binom{n+d-3}{n-2}.
\end{gather}
Spherical harmonics are a complete orthogonal system on $L^2\big(\Sdmone\big)$ with respect to the inner product
\begin{gather} \label{eq:inner_product_sdmone}
[ f, g ]_{\Sdmone} = \frac{1}{\sigma_d} \int_{\Sdmone} f(x)g(x) {\rm d}\sigma(x),
\end{gather}
where
\begin{gather*}
\sigma_d = \frac{2\pi^{\frac d2}}{\Gamma\left(\frac d2\right)} =\frac{2\pi^{\lambda+1}}{\Gamma(\lambda+1)}
\end{gather*}
is the surface area of $\Sdmone$ (for $\lambda$ see \eqref{eq:addition_formula} below).

For the set $\{ Y_{j}^n \colon 1 \leq j \leq N_{d,n} \}$, being an orthonormal basis for $\CH^d_n$, the addition formula
\begin{gather} \label{eq:addition_formula}
\sum_{j=1}^{N_{d,n}} Y_{j}^n(x) Y_{j}^n(y) =\frac{n+\lambda}{\lambda} C^\lambda_n\big( x^{\rm T} y\big)= N_{d,n}W^\lambda_n\big( x^{\rm T} y\big),\qquad \lambda = \frac{d-2}2,
\end{gather}
shows that the reproducing kernels of the spaces $\CH^d_n$ are the zonal polynomials $W_n^\lambda$ which are Gegenbauer polynomials normalized so that $W^\lambda_n(1)=1$. The \emph{Gegenbauer polynomial} of order $\lambda\geq 0$ and degree $n\in\NN_0$ is defined as the hypergeometric polynomial
\begin{gather*}
C_n^\lambda(x) = \frac{\Gamma(n+2\lambda)}{n!\Gamma(2\lambda)}\, \HypFunTO{-n}{n+2\lambda}{\lambda+\frac 12}{\frac{1-x}2}, \qquad x\in[-1,1].
\end{gather*}
The Gegenbauer polynomials are orthogonal with respect to the inner product
\begin{gather}\label{eq:inner_product}
[f,g]_\lambda = \int_{-1}^1 f(x) g(x) \big(1-x^2\big)^{\lambda-\frac 12} {\rm d}x.
\end{gather}
Indeed,
\begin{gather} \label{eq:orthogonality}
\int_{-1}^1 C_n^\lambda(x) C_m^\lambda(x)\big(1-x^2\big)^{\lambda-\frac 12} {\rm d}x =h_n^\lambda \delta_{nm},
\end{gather}
where
\begin{gather*}
h_n^\lambda =\frac{\pi\Gamma(2\lambda+n)} {2^{2\lambda-1}n!(\lambda+n)\Gamma^2(\lambda)}=\frac{\sigma_d}{\sigma_{d-1}}\frac{\lambda}{\lambda+n} C_n^\lambda(1).
\end{gather*}
Note that since the convolution of zonal functions on the sphere remains zonal, the inner pro\-duct~\eqref{eq:inner_product_sdmone} reduces to~\eqref{eq:inner_product} for zonal functions, where $\lambda=\frac{d-2}2$. The remaining weight function has the integral
\begin{gather*}
\int_{-1}^1 \big(1-x^2\big)^{\lambda-\frac 12} {\rm d}x= h^\lambda_0 = \frac{\sigma_d}{\sigma_{d-1}} C^\lambda_0(1).
\end{gather*}
Although, Gegenbauer polynomials provide a complete orthogonal system for all $\lambda>-\frac 12$, we will fix $\lambda=\frac{d-2}2$ throughout this paper.

The constant $C_n^\lambda(1)$ relates to the dimension $N_{d,n}$ given in \eqref{eq:Ndn}; indeed,
\begin{gather*}
C_n^\lambda(1) =\binom{n+d-3}{d-3}=\frac{\lambda}{\lambda+n} N_{d,n}.
\end{gather*}
Since $\{Y_{j}^n\colon 1\leq j \leq N_{d,n},\, n\in\NN_0 \}$ is a complete orthonormal system for
$L^2\big(\Sdmone\big)$, we can consider Fourier series
\begin{gather*}
f ~ \sim \sum_{n=0}^\infty \sum_{j=1}^{N_{d,n}} a_{nj}Y_{j}^n ,
 \end{gather*}
where
\begin{gather*}a_{nj}= \langle f, Y_{j}^n\rangle ,
\end{gather*}
converges to $f\in L^2\big(\Sdmone\big)$ in the $L^2$-sense.

Let $\CF^d_m$ be the subspace of $L^2\big(\Sdmone\big)$ formed by the functions $f\in L^2\big(\Sdmone\big)$ such that
 \begin{gather*}
 \sum_{n=1}^\infty [n(n+d-2)]^m \sum_{j=1}^{N_{d,n}} a_{nj}^2 < \infty.
\end{gather*}
Further, $\CF^{d,\ell}_m$ is the space of all functions $f\in \CF^d_m$ with Fourier coefficients $a_{nj}=0$, for all $1\leq j \leq N_{d,n}$ and $0\leq n\leq \ell$. In what follows we will consider approximations $s$ to $f$ whose smoothness is measured by an inner product on $\CF^{d,\ell}_m$ with an additional spherical polynomial part of degree $\ell$ viewed as a trend. The case most frequently occurring in the literature is that of~$\CF^{d,0}_m$.

For $f,g \in \CF^{d,0}_m$ with $f \sim \sum\limits_{n=1}^\infty \sum\limits_{j=1}^{N_{d,n}} a_{nj} Y^n_j$ and $g \sim \sum\limits_{n=1}^\infty \sum\limits_{j=1}^{N_{d,n}} b_{nj} Y^n_j$ and $m$ even,
 \begin{gather}
 \frac{1}{\sigma_d} \int_{\Sdmone}
 \big(\triangle_{0,d}^{m/2} f\big)(x) \big(\triangle_{0,d}^{m/2} g\big)(x) {\rm d}\sigma(x) \label{eq:semi_inner_energy_form} \\
 = \frac{1}{\sigma_d} \int_{\Sdmone} \! \left( \sum_{n=1}^\infty \sum_{j=1}^{N_{d,n}} [n(n+d-2)]^{m/2}
 a_{nj}Y^n_j(x)\right)\!\left( \sum_{n=1}^\infty \sum_{j=1}^{N_{d,n}} [n(n+d-2)]^{m/2} b_{nj}Y^n_j(x)\right)\!
{\rm d}x \nonumber \\
= \sum_{n=1}^\infty [n(n+d-2)]^{m} \sum_{j=1}^{N_{d,n}} a_{nj} b_{nj},\label{eq:energy_semi_inner_coeff_form}
\end{gather}
by the extended Parseval identity in the space $L^2\big(\Sdmone\big)$. In view of the equality between expressions~\eqref{eq:semi_inner_energy_form} and~\eqref{eq:energy_semi_inner_coeff_form} define an ``energy'' semi-inner product for $\CF^{d}_m$ by
\begin{gather}
(f,g)_{m,\ell}= \sum_{n=\ell+1}^\infty [n(n+d-2)]^{m} \sum_{j=1}^{N_{d,n}} a_{nj} b_{nj} . \label{eq:energy_semi_inner_prod}
\end{gather}
It is clear from \eqref{eq:semi_inner_energy_form} and \eqref{eq:energy_semi_inner_coeff_form} that this is an analogue of the usual semi-inner product associated with smoothing splines on~$\RR^d$.

$(\cdot,\cdot)_{m,\ell}$ is an inner product for $\CF^{d,\ell}_m$. It is easy to show that $\CF^{d,\ell}_m$ with norm $\|f\|_{m,\ell}=\sqrt{(f,f)_{m,\ell}}$ is a Hilbert space. A proof could be based on the arguments in \cite[pp.~247--250]{CheneyLight_2000}.

Now for $x,y\in \Sdmone$ define
 \begin{gather}
K_{d,m,\ell}(x,y) =\sum_{n=\ell+1}^\infty [n(n+d-2)]^{-m} \sum_{j=1}^{N_{d,n}} Y^n_j(x)Y^n_j(y) \nonumber \\
\hphantom{K_{d,m,\ell}(x,y}{} = \sum_{n=\ell+1}^\infty [n(n+d-2)]^{-m} N_{d,n}W^\lambda_n\big(x^{\rm T} y\big), \label{series_for_K}
 \end{gather}
by the addition formula \eqref{eq:addition_formula}. $K_{d,m,\ell}$ is clearly a zonal kernel since it depends only on the cosine of the angle between $x$ and $y$. We will use the notation $K_{d,m}$ for the kernels $K_{d,m,0}$, and call these kernels the {\em thinplate spline kernels for the sphere} $\bbS^{d-1}$. The~$K_{2,m}$ and~$K_{3,m}$ kernels were initially introduced by Wahba~\cite{Wahba81}. Since $K_{d,m,\ell}(x,y)$ is zonal we can sensibly define functions~$k_{d,m,\ell}$ by
 \begin{gather} \label{eq:kernels_and_functions}
 k_{d,m,\ell}(\xi)= K_{d,m,\ell}(x,y), \qquad \text{where}\quad x,y \in \bbS^{d-1}\quad \text{and}\quad \xi=x^{\rm T} y.
 \end{gather}
 Explicitly,
 \begin{gather}
 k_{d,m,\ell}(\xi) = \sum_{n=\ell+1}^\infty [n(n+d-2)]^{-m} N_{d,n}W^\lambda_n(\xi), \label{series_for_k}
 \end{gather}
We will refer to the functions $k_{d,m,\ell}$ as {\em semi-reproducing functions}, and the functions $k_{d,m}=k_{d,m,0}$ as {\em thinplate
 spline functions}.

\begin{Lemma}\label{Le:rk_for_truncFS} Let $\ell$ be a non-negative integer and $2m \geq d \geq 2$.
 \begin{itemize}\itemsep=0pt
 \item[$(a)$] $K_{d,m,\ell}(x,y)$ is the reproducing kernel for the Hilbert space $\CF^{d,\ell}_m$.
 \item[$(b)$] $K_{d,m,\ell}(x,y)$ considered as a function in $C\big(\bbS^{d-1} \times \bbS^{d-1}\big)$ is strictly positive definite.
 \end{itemize}
 \end{Lemma}
 \begin{proof} {\em Proof of part $(a)$.} Assume $2m\geq d \geq 2$. To show $K=K_{d,m,\ell}$ is the reproducing kernel for $\CF^{d,\ell}_m$ it suffices to show that the following two properties hold, \cite[p.~317]{Davis75}.
 \begin{itemize}\itemsep=0pt
 \item[(i)] For each fixed $y \in \Sdmone$, $ f(\cdot)=K(\cdot,y)$ is in $\CF^{d,\ell}_m$.
 \item[(ii)] For each function $f \in \CF^{d,\ell}_m$ the reproducing property
 \begin{gather*}
\big( f(\cdot), K(\cdot,y) \big)_{m,\ell} =f(y),
 \end{gather*}
 holds.
 \end{itemize}

To show property (i) let $y$ be some fixed point in $\Sdmone$ and define $k_y(\cdot)=K(\cdot,y)$. From the definition of $K$, $k_y$ has Fourier coefficients
 \begin{gather}
\label{eq:fc_kernel}
 a_{nj}=[n(n+d-2)]^{-m} Y^n_j(y), \qquad n \geq \ell+1 \quad \text{and}\quad 1\leq j \leq N_{d,n}.
 \end{gather}
 Hence,
 \begin{gather*}
(k_y,k_y)_{m,\ell} = \sum_{n=\ell+1}^\infty [n(n+d-2)]^m \sum_{j=1}^{N_{d,n}} (a_{nj})^2 = \sum_{n=\ell+1}^\infty [n(n+d-2)]^{-m} \sum_{j=1}^{N_{d,n}} \left( Y^n_j(y) \right)^2 \\
\hphantom{(k_y,k_y)_{m,\ell}}{} = \sum_{n=\ell+1}^\infty [n(n+d-2)]^{-m} N_{d,n} W^\lambda_n(1)=\sum_{n=\ell+1}^\infty [n(n+d-2)]^{-m} N_{d,n} ,
 \end{gather*}
 where in the second to last step the addition formula~\eqref{eq:addition_formula} has been used. The estimate $N_{d,n}={\mathcal O}\big(n^{d-2}\big)$, holding for $d>1$, shows that the sum above is finite when $2m\geq d$, and hence $k_y \in \CF^{d,\ell}_m$ as required.

To show the reproducing property let $f$ be any function in $\CF^{d,\ell}_m$. Suppose $f$ has Fourier series $ \sum\limits_{n=\ell+1}^\infty \sum\limits_{j=1}^{N_{d,n}} a_{nj} Y^n_j$. Then, from \eqref{eq:energy_semi_inner_prod} and \eqref{eq:fc_kernel}
\begin{gather*}
 ( f, k_y )_{m,\ell} = \sum_{n=\ell+1}^\infty [n(n+d-2)]^{m} \sum_{j=1}^{N_{d,n}} a_{nj} [n(n+d-2)]^{-m} Y^n_j(y) \\
\hphantom{( f, k_y )_{m,\ell}}{} = \sum_{n=\ell+1}^\infty \sum_{j=1}^{N_{d,n}} a_{nj} Y^n_j(y) = f(y). 
\end{gather*}
That is the reproducing property holds.

{\em Proof of part $(b)$.} In view of the characterisations of strict positive definiteness of zonal kernels given by Chen, Menegatto and Sun~\cite{Ch03} for $d>2$, and by Menegatto~\cite{Me06} for $d=2$, part~(b) follows from the signs of the Gegenbauer coefficients of~$K_{d,m}$ displayed in equation~\eqref{series_for_K}.
\end{proof}

Above we have shown $K_{d,m,\ell}$ is the reproducing kernel for the space $\CF^{d,\ell}_m$ which arises from using a Fourier projection onto spherical polynomials to split the space $\CF^d_m$ into a direct sum $\CH_0 \oplus\CF^{d,\ell}_m$, where $\CH_0 = \cup_{n=0}^\ell\CH^d_n$ is the space of spherical polynomials of degree at most~$\ell$.

For interpolation problems it is natural to use instead a projection onto polynomials interpolating at a certain finite set of points, and this results in a different direct sum decomposition. Fortunately, here the semi-reproducing kernel approach becomes especially convenient as $K_{d,m,\ell}$ is a semi-reproducing kernel for the space $\CF^{d,\ell}_m$ with respect to the space of polynomials~$\CH_0$. This is the content of the following easily shown lemma whose proof is included for the sake of completeness.

\begin{Lemma}\label{Le:rk_for_semi-inner_product} Let $\ell$ be a nonnegative integer and $2m \geq d \geq 2$. The reproducing kernel $K_{d,m,\ell}$ for the space $\CF^{d,\ell}_m$ with semi-inner product $(\cdot, \cdot)_{m,\ell}$, defined above, is one choice of semi-reproducing kernel for the space $\CF^{d}_m$ with respect to the space of spherical polynomials of degree $\ell$, $\CH_0 = \cup_{n=0}^\ell \CH^d_n$.
\end{Lemma}
\begin{proof}$\CF^{d}_m = \CH $ is considered as a semi-inner product space with semi-inner product $(\cdot,\cdot)_{m,\ell}$. Choose $P_0$ as Fourier projection onto the kernel $\CH_0$ of the semi-inner product, which is the space of spherical polynomials of degree not exceeding~$\ell$. Set $\CH_1=(I-P_0) \CF^{d}_m=\CF^{d,\ell}_m$. Then clearly $\CH= \CF^{d}_m = \CH_0\oplus \CF^{d,\ell}_m = \CH_0\oplus \CH_1$. Also, for each $x\in\Sdmone$, $K_{d,m,\ell}( \cdot,x) \in \CH_1=\CF^{d,\ell}_m \subset \CF^{d}_m=\CH$ by Lemma~\ref{Le:rk_for_truncFS}(a). Therefore, considering an $\CH_0$-increment $(\CX, \va)$
\begin{gather*}
 \sum_{i=1}^n a_i K_{d,m,\ell}(\cdot,x_i) \in \CF^{d,\ell}_m \subset \CF^{d}_m,
\end{gather*}
the first property, i.e., property~\eqref{eq:first_prop_semi_repro_kernel}, of a semi-reproducing kernel. Also, given any $f\in K^d_m$ the direct sum splitting allows us to write $f=f_0+f_1$ where $f_0 \in \CH_0$ and $f_1 \in \CF^{d,\ell}_m$. Therefore, considering the $\CH_0$-increment $(\CX,\va)$
\begin{gather*}
\left( f(\cdot), \sum_{i=1}^n a_i K_{d,m,\ell}(\cdot,x_i) \right)_{m,\ell}=
\left( f_0(\cdot)+f_1(\cdot), \sum_{i=1}^n a_i K_{d,m,\ell}(\cdot,x_i) \right)_{m,\ell} \\
=\sum_{i=1}^n a_i \big( f_0(\cdot), K_{d,m,\ell}(\cdot,x_i) \big)_{m,\ell}
 +\sum_{i=1}^n a_i \left( f_1(\cdot), \sum_{i=1}^n K_{d,m,\ell}(\cdot,x_i) \right)_{m,\ell} = 0 +\sum_{i=1}^n a_i f_1(x_i) ,
 \end{gather*}
 which follows from $( h_0, h_1 )_{m,\ell}=0$ for all $h_0\in \CH_0$ and $h_1 \in \CF^{d,\ell}_m$, and also from $K_{d,m,\ell}$ being the reproducing kernel for $\CF^{d,\ell}_m$. Continuing, using the vanishing property of $\CH_0$-increments,
\begin{align*}
\left( f(\cdot), \sum_{i=1}^n a_i K_{d,m,\ell}(\cdot,x_i) \right)_{m,\ell}
 &= 0 +\sum_{i=1}^n a_i f_1(x_i)\\
 &= \sum_{i=1}^n a_i f_0(x_i) + \sum_{i=1}^n a_i f_1(x_i) = \sum_{i=1}^n a_i f(x_i),
\end{align*}
the second property, i.e., property~\eqref{eq:second_prop_semi_repro_kernel}, of a semi-reproducing kernel. Therefore, $K_{d,m,\ell}$ is a~semi-reproducing kernel for $\CF^d_m$ with semi-inner product $(\cdot,\cdot)_{m,\ell}$, as required.
\end{proof}

In view of Lemma~\ref{Le:rk_for_semi-inner_product}, Theorem~\ref{thm:Interpolation_and_smoothing} concerning the solution of interpolation and penalized least squares fitting problems applies to interpolation and smoothing problems on the sphere. Applying the theorem the semi-reproducing kernel $K=K_{d,m,\ell}$, defined in equation~\eqref{series_for_K}, plays a central role in interpolation and penalized least squares fitting problems posed in the space $\CH=\CF^d_m$, with semi-inner product $(\cdot,\cdot)_{m,\ell}$ with respect to the finite dimensional subspace $ \CH_0 = \cup_{n=0}^\ell \CH^d_n$. In this section we have given series expansions for these kernels. In Section~\ref{sec:recurrence}, Theorem~\ref{thm:recurrence_for_TPS}, below, we will provide a recurrence relation for the particularly important thinplate spline kernels, $K_{d,m,}$, via a recurrence for the corresponding functions~$k_{d,m}$. In Sections~\ref{sec:explicit_forms} and~\ref{sec:m_equal_1} the recurrence relation will be used to give short explicit expressions for many of the thinplate spline kernels.

\section[The operator $T$ and its adjoint $T^*$]{The operator $\boldsymbol{T}$ and its adjoint $\boldsymbol{T^*}$}\label{sec:the_operators_T_and_T*}

In this section we discuss an operator $T$, and its adjoint $T^*$, which will be crucial parts of the recurrence for the thinplate spline functions~$k_{d,m}$. These operators were defined by Martinez-Morales in~\cite{Ma05}.

In view of the series expansions for the kernels $k_{d,m}$ given in equation~\eqref{series_for_k} a multiplier opera\-tor with Fourier multiplier of $(n(n+2\lambda))^{-1}$ would transform $k_{d,m}$ into $k_{d,m+1}$. The operators~$T$ and~$T^*$, are discussed below have some, but not quite all, the desired properties.

Note that the differential equation for the Gegenbauer polynomials is given by \cite[equation~(22.6.5)]{Abramowitz} or \cite[Table~18.8.1]{DLMF}
\begin{gather*}
\big(1-x^2\big)y'' -(2\lambda+1)x y' + n(n+2\lambda)y = 0.
\end{gather*}
Rewriting this we obtain
\begin{gather*}
\big(1-x^2\big)^{-\lambda+\frac 12} \frac{{\rm d}}{{\rm d}x} \left( \big(1-x^2\big)^{\lambda+\frac 12} \frac{{\rm d}}{{\rm d}x} \right) y = -n(n+2\lambda) y.
\end{gather*}

Call the operator on the left hand side of the equation $D_\lambda$. Then, for $\lambda>0$ and $n\in{\mathbb N}$,
\begin{gather*}
D_\lambda C_n^\lambda(x) = -n(n+2\lambda) C_n^\lambda(x).
\end{gather*}
Let us formally invert the equation $D_\lambda f = g$ using averages centered at~1. Dealing with the outer derivative in $D_\lambda$ we obtain that
\begin{gather*}
\int_x^1 \big(1-y^2\big)^{\lambda-\frac 12} g(y) {\rm d}y = \left. \big(1-y^2\big)^{\lambda+\frac 12}\frac{{\rm d}}{{\rm d}y}f(y) \right|_x^1.
\end{gather*}
If
\begin{gather}\label{eq:condition1}
f'(y) \ \text{is continuous at}\ y=1,
\end{gather}
the term on the right hand side for $y=1$ vanishes. We can then proceed obtaining
\begin{gather*}
 \int_x^1 \big(1-y^2\big)^{-\lambda-\frac 12} \int_y^1 \big(1-z^2\big)^{\lambda-\frac 12} g(z) {\rm d}z {\rm d}y =f(x) - f(1).
\end{gather*}
Setting $f(x)=C_n^\lambda(x)$ and $g(x)=-n(n+2\lambda)C_n^\lambda(x)$, condition (\ref{eq:condition1}) is clearly satisfied. We therefore obtain the following statement.

\begin{Proposition}\label{prop:eigenfunction_Tadjoint}For $\lambda>0$ and $n\in{\mathbb N}$
\begin{gather*}
\int_x^1 \big(1-y^2\big)^{-\lambda-\frac 12} \int_y^1 \big(1-z^2\big)^{\lambda-\frac 12} C_n^\lambda(z) {\rm d}z {\rm d}y = \frac{C_n^\lambda(1)-C_n^{\lambda}(x)}{n(n+2\lambda)}.
\end{gather*}
\end{Proposition}

\begin{proof}We would like to give a direct proof of the statement. Towards this goal, we use an integral given in \cite[equation~(22.13.2)]{Abramowitz} (or \cite[equation~(18.17.1)]{DLMF} for the general Jacobi case)
\begin{gather*}
\int_0^x \big(1-y^2\big)^{\lambda-\frac 12} C_n^\lambda(y) {\rm d}y =\frac{2\lambda}{n(2\lambda+n)}\big[ C_{n-1}^{\lambda+1}(0) - \big(1-x^2\big)^{\lambda+\frac 12}C_{n-1}^{\lambda+1}(x) \big].
\end{gather*}
Decomposing the integral over $[0,1]$ into two integrals over $[0,x]$ and $[x,1]$, respectively, we can use the formula to obtain that
\begin{gather*}
\int_y^1 \big(1-z^2\big)^{\lambda-\frac 12} C_n^\lambda(z) {\rm d}z =
\frac{2\lambda}{n(2\lambda+n)} \big(1-y^2\big)^{\lambda+\frac 12} C_{n-1}^{\lambda+1}(y).
\end{gather*}
Note that the function on the right hand side vanishes at $y=1$. Towards the claim, it remains to integrate the polynomial $C_{n-1}^{\lambda+1}$ which readily follows from \cite[equation~(18.9.19)]{DLMF}
\begin{gather*}
\frac{{\rm d}}{{\rm d}x} C_n^\lambda(x) = 2\lambda C_{n-1}^{\lambda+1}(x)\quad \Leftrightarrow \quad C_n^{\lambda}(x) = 2\lambda \int C_{n-1}^{\lambda+1}(y) {\rm d}y,
\end{gather*}
completing the proof.
\end{proof}

Similarly, we could have treated the average centered at $-1$ giving the following result.
\begin{Proposition}\label{prop:eigenfunction_T}
For $\lambda>0$ and $n\in{\mathbb N}$ we have that
\begin{gather*}
\int_{-1}^x \big(1-y^2\big)^{-\lambda-\frac 12} \int_{-1}^y \big(1-z^2\big)^{\lambda-\frac 12} C_n^\lambda(z) {\rm d}z {\rm d}y =\frac{C_n^\lambda(-1)-C_n^{\lambda}(x)}{n(n+2\lambda)}.
\end{gather*}
\end{Proposition}
The operators
\begin{gather*}
T_\lambda f(x) = - \int_{-1}^x \big(1-y^2\big)^{-\lambda-\frac 12} \int_{-1}^y \big(1-z^2\big)^{\lambda-\frac 12} f(z) {\rm d}z {\rm d}y
\end{gather*}
and
\begin{gather*}
T_\lambda^*f(x) = - \int_{x}^1 \big(1-y^2\big)^{-\lambda-\frac 12} \int_{y}^1 \big(1-z^2\big)^{\lambda-\frac 12} f(z) {\rm d}z {\rm d}y
\end{gather*}
have been defined in \cite{Ma05}, showing that $T_\lambda^*$ is the adjoint of $T_\lambda$ with respect to the inner pro\-duct~(\ref{eq:inner_product}). To be precise, the following statement holds true (cf.\ \cite[Theorem~3]{Ma05}).

\begin{Theorem}\label{thm:Morales_theorem3} Let $f\in C[-1,1)\cap L^1_\lambda[-1,1]$ and $g\in C(-1,1]\cap L^1_\lambda[-1,1]$. Then $gT_\lambda f, fT_\lambda^*g\in L^1_\lambda[-1,1]$ and
\begin{gather*}
[T_\lambda f,g]_\lambda = [f,T_\lambda^*g]_\lambda.
\end{gather*}
\end{Theorem}
The proof exploits the fact that with the weight $\big(1-z^2\big)^{\lambda-\frac 12}$ in the inner integral, continuity of $f$ suffices to cope with the singularity introduced by the weight $\big(1-y^2\big)^{-\lambda-\frac 12}$ in the outer integral of~$T_\lambda$. Based on this observation, the theorem follows from Fubini's theorem.

Propositions~\ref{prop:eigenfunction_Tadjoint} and~\ref{prop:eigenfunction_T} above thus show that the polynomials $C_n^\lambda$ are basically~-- up to a~constant~-- eigenfunctions of $T^*$ and $T$, respectively. This is somewhat obvious from the fact that both $T$ and $T^*$ invert $D_\lambda$. Both propositions have been derived in \cite[Lemma~1]{Ma05} via a~different proof employing the Rodriguez formula of the Gegenbauer polynomials.

\section{A recurrence for the thinplate spline functions for the sphere} \label{sec:recurrence}

This section concerns a recurrence for the thinplate spline functions $k_{d,m}$ for $\bbS^{d-1}$. The recurrence will be used in Sections~\ref{sec:explicit_forms} and~\ref{sec:m_equal_1} to give short explicit forms for many of the functions $k_{d,m}$. It is important to note that the corresponding recurrence given in \cite[Theorem~4]{Ma05} is incorrect and does not yield the thinplate spline functions $k_{d,m}$. Consequently, many of the explicit formulas claimed for the functions $k_{d,m}$ in the paper~\cite{Ma05} are also incorrect.

\begin{Theorem}\label{thm:recurrence_for_TPS}
Let $d\geq 2$, $\lambda=\frac{d-2}2$, $x\in[-1,1]$, and $e_0(x)=1$ for all $x\in[-1,1]$. The thinplate spline functions, $k_{d,m}$, $m\in\bbN$, for the sphere $\cS^{d-1}$, defined via the series~\eqref{series_for_k}, are alternatively generated by the recurrence
\begin{gather*}
k_{d,m}(x) = \begin{cases} \dfrac{[ e_0, T_\lambda e_0]_\lambda}{[ e_0,e_0]_\lambda} -(T_\lambda e_0)(x) ,&\text{when}\ m=1, \\
 (T_\lambda k_{d,m-1} )(x) - \dfrac{[ e_0, T_\lambda k_{d,m-1}]_\lambda}{[ e_0,e_0]_\lambda}, & \text{when}\ m>1.
\end{cases}
\end{gather*}
\end{Theorem}

\begin{proof}Fix $d\geq 2$ and define a sequence of functions $(f_1, f_2, \ldots)$ by the recurrence
\begin{gather} \label{eq:recursion_stated_for_f}
f_{m}(x) = \begin{cases} \dfrac{[ e_0, T_\lambda e_0]_\lambda}{[ e_0,e_0]_\lambda} -(T_\lambda e_0)(x) ,&\text{when}\ m=1,\\
( T_\lambda f_{m-1})(x) - \dfrac{[ e_0, T_\lambda f_{m-1}]_\lambda}{[ e_0,e_0]_\lambda}, & \text{when}\ m>1,
\end{cases}
\end{gather}
this recursion mirroring the one in the statement of the theorem.

Throughout this proof we view the series definition \eqref{series_for_k} as an orthogonal expansion
\begin{gather*}
g \sim \sum_{n=0}^\infty \widehat{g}_n C^\lambda_n
\end{gather*}
in terms of the Gegenbauer polynomials $C^\lambda_n$. The uniqueness theorem tells us that functions with the same coefficients are identical.

First, observe that the constant term in the definition of $f_m$ ensures that the zeroth Fourier coefficient $\widehat{(f_m)}_0$ is zero. Therefore, consider in what follows Fourier coefficients of index $n\geq 1$.

Clearly, $e_0\in C[-1,1]\cap L^1_\lambda[-1,1]$. It therefore follows from Theorem~\ref{thm:Morales_theorem3} that $f_1$ is conti\-nuous on $[-1,1]$ and $f_1\in L^1_\lambda[-1,1]$. Using induction on $m$ we can then conclude, again using Theorem~\ref{thm:Morales_theorem3}, that~$f_m$ is continuous on $[-1,1]$ and $f_m\in L^1_\lambda[-1,1]$ for $m>1$.

For a function $f\in L^1_\lambda[-1,1]$, and $n\geq 1$, Theorem~\ref{thm:Morales_theorem3} gives that
\begin{gather*}
\widehat{T_\lambda f}_n = \frac 1{h_n^\lambda}\int_{-1}^1 T_\lambda f(x) C_n^\lambda(x) \big(1-x^2\big)^{\lambda-\frac 12} {\rm d}x
 = \frac 1{h_n^\lambda} \int_{-1}^1 f(x) T_\lambda^* C_n^\lambda(x) \big(1-x^2\big)^{\lambda-\frac 12} {\rm d}x,
\end{gather*}
which by Proposition~\ref{prop:eigenfunction_Tadjoint} yields
\begin{gather}
\widehat{T_\lambda f}_n = \frac 1{n(n+2\lambda)} \frac{1}{h_n^\lambda} \int_{-1}^1 f(x) C_n^{\lambda}(x)\big(1-x^2\big)^{\lambda-\frac 12} {\rm d}x\nonumber\\
\hphantom{\widehat{T_\lambda f}_n=}{} - \frac 1{n(n+2\lambda)} \frac{C_n^\lambda(1)}{h_n^\lambda}
\int_{-1}^1 f(x) \big(1-x^2\big)^{\lambda-\frac 12} {\rm d}x. \label{eq:critical_Fourier_coeff}
\end{gather}

In the special case of $f=e_0$, the first integral vanishes due to orthogonality~(\ref{eq:orthogonality}). Furthermore,
\begin{gather*}
\int_{-1}^1 e_0(x)\big(1-x^2\big)^{\lambda-\frac 12} {\rm d}x= \int_{-1}^1 \big(1-x^2\big)^{\lambda-\frac 12} {\rm d}x = [ e_0, e_0 ]_\lambda = h_0^\lambda.
\end{gather*}
We therefore obtain that for $n\geq 1$,
\begin{gather*}
\widehat{(T_\lambda e_0)}_n = - \frac{1}{n(n+2\lambda)} \frac{C_n^\lambda(1)}{h_n^\lambda} h^\lambda_0 = -\frac{1}{n(n+2\lambda)} \frac{n+\lambda}\lambda = -\frac{1}{n(n+2\lambda)} \frac{N_{d,n}}{C_n^\lambda(1)}.
\end{gather*}
Thus, the function $f_1$ generated as specified in equation~\eqref{eq:recursion_stated_for_f} has the same Fourier coefficients as the thinplate spline function $k_{d,1}$ of equation~\eqref{series_for_k}. Thus, by uniqueness, it is~$k_{d,1}$.

Now returning to general functions $f$ we can rewrite \eqref{eq:critical_Fourier_coeff} as
\begin{gather}\label{eq:fourier_coefficient_of_Tf}
\widehat{T_\lambda f}_n = \frac 1{n(n+2\lambda)} {\widehat f}_n- \frac 1{n(n+2\lambda)} C_n^\lambda(1)\frac{h^\lambda_0}{h_n^\lambda} {\widehat f}_0.
\end{gather}

From (\ref{eq:fourier_coefficient_of_Tf}) and the definition of~$f_m$ we have that for $m\geq 2$ and $n \geq 1$
\begin{gather*}
\big(\widehat{f_m}\big)_n = \frac{1}{n(n+2\lambda)} \big(\widehat{f_{m-1}}\big)_n -\frac 1{n(n+2\lambda)} C_n^\lambda(1)\frac{h_0^\lambda}{h_n^\lambda} \big(\widehat{f_{m-1}}\big)_0= \frac{1}{n(n+2\lambda)} \big(\widehat{f_{m-1}}\big)_n,
\end{gather*}
since $\big(\widehat{f_{m-1}}\big)_0=0$. Thus, by induction
\begin{gather*}\big(\widehat{f_m}\big)_n =
\begin{cases}
0, & \text{when}\ n=0,\\
 (n(n+2\lambda) )^{-m} \dfrac{N_{d,n}}{C^\lambda_n(1)}, & \text{when} \ n \geq 1.
\end{cases}
\end{gather*}
Hence $f_m$ has the same Fourier coefficients as $k_{d,m}$. Therefore, by uniqueness, it is~$k_{d,m}$. That is the thinplate spline functions,
$k_{d,m}$, are generated by the recursion of the theorem. \end{proof}

\section[Explicit forms for some of the thinplate spline functions $k_{d,m}$]{Explicit forms for some of the thinplate spline functions $\boldsymbol{k_{d,m}}$}\label{sec:explicit_forms}

The reader will recall the correspondence between the zonal thinplate spline kernels $K_{d,m}$ and the associated functions $k_{d,m}$, see~\eqref{eq:kernels_and_functions}. The recurrences of the previous section yield explicit formulas for many of the thinplate spline functions $k_{d,m}$. A sample of these explicit expressions is presented below, thereby correcting the expressions given in~\cite{Ma05}.

In the formulas that follow $u= \frac{1-x}{2}$ and $v = \frac{\pi}{2}+\arcsin(x) $. Note that in angular coordinates $x=\cos\theta$ and $v=\pi-\theta$.
\subsection[Functions for $\bbS^1$]{Functions for $\boldsymbol{\bbS^1}$}
\begin{gather*}
k_{2,1} = \frac{1}{2}v^2 -\frac{1}{6}\pi^2, \qquad
k_{2,2} = -\frac{1}{24} v^4 + \frac{1}{12}\pi^2 v^2 -\frac{7}{360}\pi^4, \\
k_{2,3} =\frac{1}{720}v^6 -\frac{1}{144} \pi^2 v^4 + \frac{7}{720}\pi^4 v^2 -\frac{31}{15120}\pi^6,\\
k_{2,4}= -\frac{1}{40320}v^8 +\frac{1}{4320}\pi^2 v^6 -\frac{7}{8640}\pi^4 v^4 +\frac{31}{30240}\pi^6v^2 -\frac{127}{60480}\pi^8.
\end{gather*}
See Wahba~\cite[p.~22]{Wabha90} for explicit forms of these functions in terms of Bernoulli polynomials.

 \subsection[Functions for $\bbS^2$]{Functions for $\boldsymbol{\bbS^2}$}
\begin{gather*}
k_{3,1} =-\ln (u) -1, \qquad k_{3,2} = \Li_2(1-u) +1 -\frac{\pi^2}{6}, \\
k_{3,3} = -2\Li_3(u) -\Li_2(1-u)+\ln (u)\Li_2(u) + 2\zeta(3)+\frac{\pi^2}{6}-2.
\end{gather*}

\subsection[Functions for $\bbS^3$]{Functions for $\boldsymbol{\bbS^3}$}
\begin{gather*}
k_{4,1} =\frac{1}{2}\frac{x v}{\sqrt{1-x^2}} -\frac{1}{4}, \qquad k_{4,2} = \frac{1}{8}v^2 +\frac{1}{16}-\frac{\pi^2}{24}.
\end{gather*}
\subsection[Functions for $\bbS^4$]{Functions for $\boldsymbol{\bbS^4}$}
\begin{gather*}
k_{5,1} = -\frac{1}{3}\ln(u) +\frac{1}{6u} -\frac{7}{9}, \qquad
k_{5,2} = \frac{1}{9}\Li_2(1-u) -\frac{2}{9}\ln(u) +\frac{\ln(u)}{9(x+1)} +\frac{1}{81}-\frac{\pi^2}{54}.
\end{gather*}
 \subsection{Some functions for higher dimensional spheres}
 \begin{gather*}
k_{6,1} = xv\left( \frac{1}{4\sqrt{1-x^2} }+\frac{1}{8\big(1-x^2\big)^{3/2}} \right) +\frac{1}{8\big(1-x^2\big)} - \frac{5}{16},\\
k_{8,1} = xv \left( \frac{1}{6\sqrt{1-x^2}} +\frac{1}{12\big(1-x^2\big)^{3/2}}+\frac{1}{16\big(1-x^2\big)^{5/2}} \right)\\
\hphantom{k_{8,1} =}{} +\frac{1}{16\big(1-x^2\big)} +\frac{1}{16\big(1-x^2\big)^2} -\frac{5}{18},
\end{gather*}
and in general, as is shown in Section~\ref{subsection8_1} below, for $d=2\lambda+2$ even, i.e., when $\lambda$ is an integer,
\begin{gather*}
 k_{2\lambda+2,1} (x) = x v \sum_{j=1}^\lambda c_j^\lambda \big(1 -x^2\big)^{-j+\frac{1}{2}}+ \sum_{j=1}^{\lambda-1} d_j^\lambda \big(1 -x^2\big)^{-j} -C_\lambda .
\end{gather*}

Also,
 \begin{gather*}
k_{7,1} = -\frac{1}{5} \ln(u) +\frac{1}{10 u}+\frac{1}{60 u^2} -\frac{43}{75},\\
k_{9,1} = -\frac{1}{7} \ln(u) +\frac{1}{14 u} +\frac{1}{70u^2}+\frac{1}{420u^3} -\frac{337}{735},\\
k_{11,1} = -\frac{1}{9}\ln(u) +\frac{1}{18u} +\frac{1}{84u^2}+\frac{1}{378u^3}+\frac{1}{2520u^4}-\frac{1091}{2835},
\end{gather*}
and in general, as is shown in Section~\ref{subsection8_2} below, for $d=2\kappa+3$ odd,
 \begin{gather*}
 k_{2\kappa+3,1}(x) =\frac{-1}{2\kappa+1} \ln(u) + \sum_{\nu=1}^\kappa g^\lambda_\nu(1-x)^{-\nu} -D_\lambda,
\end{gather*}
with $g^\lambda_\nu$ as given in \eqref{eq:def_g_lambda_nu}.

\section[Explicit formulas for the thinplate spline functions $k_{d,1}$, $d>2$]{Explicit formulas for the thinplate spline functions $\boldsymbol{k_{d,1}}$, $\boldsymbol{d>2}$} \label{sec:m_equal_1}

In the section simple explicit formulas will be obtained for the thinplate spline functions~$k_{d,1}$. As explained at the end of Section~\ref{sec:Thinplate_splines_on_the_sphere} the function $k_{d,1}$ is associated with approximation problems on~$\Sdmone $.

Theorem~\ref{thm:recurrence_for_TPS} and the formula for the operator $T_\lambda$ given in Proposition~\ref{prop:eigenfunction_T} lead to a~method of calculating~$k_{d,1}$. Define
\begin{gather}
G_\beta^\alpha (y)= \int_{-1}^y \left(1-z^2\right)^{\beta -\half} (1-z)^{-\alpha} {\rm d}z, \qquad \text{where}\quad \alpha \geq 0\quad \text{and} \quad \beta >\alpha -\half, \label{eq:Glambda_def}
\end{gather}
$G_\beta=G_\beta^0$, and
\begin{gather*} 
 F_\beta(x) = \int_{-1}^1 \big(1-y^2\big)^{-\beta -\half } G_\beta(y) {\rm d}y.
 \end{gather*}
Then, with $\lambda = (d-2)/2$,
\begin{gather} \label{eq:express1_for_kd1}
 k_{d,1}(x) = F_\lambda (x) - \frac{[ F_\lambda, e_0]_\lambda}{[e_0,e_0]_\lambda} .
 \end{gather}

Actually it is somewhat easier to deal with the indefinite integral
\begin{gather} \label{eq:Slambda}
S_\lambda (y) = \int \big(1-y^2\big)^{-(2\lambda +1)/2} G_\lambda (y) {\rm d}y,
\end{gather}
where here we mean any fixed representative value of the indefinite integral, since the second term in \eqref{eq:express1_for_kd1} deletes the constant term. Then, from equation~\eqref{eq:express1_for_kd1},
\begin{gather} \label{eq:alternative_formula_for_k_d1}
 k_{d,1}(x) = S_\lambda (x) - \frac{[S_\lambda, e_0]_\lambda}{[e_0,e_0]_\lambda}.
 \end{gather}

 \subsection[The functions $k_{d,1}$ when $d$ is even]{The functions $\boldsymbol{k_{d,1}}$ when $\boldsymbol{d}$ is even} \label{subsection8_1}
 This subsection considers the functions $k_{d,1}$ when $d$ is even. Thus in this subsection $\lambda$ is a~positive integer.
\begin{Lemma}
\begin{gather}
G_0 (y) =\frac{\pi}{2} + \arcsin(y), \label{eq:G0}
\end{gather}
and
\begin{gather}
G_\beta(y) = \frac{1}{2\beta} y \big(1-y^2\big)^{(2\beta-1)/2} + \frac{2\beta-1}{2\beta} G_{\beta-1}(y), \qquad \beta > 1/2.\label{eq:Glambda_recur}
\end{gather}
Explicitly, for $\lambda \in \N$,
\begin{gather}
G_\lambda (y) =a_0^\lambda \left(\frac{\pi}{2} + \arcsin (y) \right) +a_1^\lambda y \big(1-y^2\big)^{1/2}+ a_2^\lambda y\big(1-y^2\big)^{3/2} + \cdots\nonumber\\
\hphantom{G_\lambda (y) =}{} +a_\lambda^\lambda y \big(1-y^2\big)^{(2\lambda-1)/2},\label{eq:Glambda_form}
\end{gather}
where
\begin{gather*} 
a_j^\lambda = \begin{cases} \dfrac{(2j-2)!!}{(2j-1)!!} \dfrac{(2\lambda-1)!!}{(2\lambda)!!}, &1 \leq j \leq \lambda,\\
a_1^\lambda,& j=0,\\
0, & \text{otherwise}.
\end{cases}
\end{gather*}
\end{Lemma}
\begin{proof} Equation~\eqref{eq:G0} follows immediately from the definition~\eqref{eq:Glambda_def}. The recurrence~\eqref{eq:Glambda_recur} follows from the definition~\eqref{eq:Glambda_def} via an easy integration by parts.

The general form of the explicit expression for $G_\lambda(y)$, given in equation~\eqref{eq:Glambda_form}, is clear from the expression for $G_0(y)$ and the recurrence~\eqref{eq:Glambda_recur}. Consider now the expression for the coefficients $a_j^\lambda$ in equation~\eqref{eq:Glambda_form}. Considering $G_j(y)$ it is clear from the form of $G_{j-1}(y)$, and the recurrence, that the coefficient $a_j^j$ of $y\big(1-y^2\big)^{(2j-1)/2}$ is $1/(2j)$. Now applying the recurrence $\lambda-j$ times to yield $G_{j+1}(y)$, $G_{j+2}(y)$, up to $G_\lambda(j)$, in turn, it follows that
\begin{gather*}
a^\lambda_j = \frac{1}{2j} \frac{2j+1}{2j+2} \frac{2j+3}{2j+4} \cdots \frac{2\lambda-1}{2\lambda}
 = \frac{1}{2j} \frac{(2\lambda-1)!!}{(2j-1)!!} \frac{(2j)!!}{(2\lambda)!!}\\
 \hphantom{a^\lambda_j}{} = \frac{(2\lambda-1)!!}{(2j-1)!!} \frac{(2j-2)!!}{(2\lambda)!!} , \qquad 1 \leq j \leq \lambda.
\end{gather*}
Further, the expression for $G_0(y)$ and the recurrence, yield the explicit expression
\begin{gather*}G_1(y) = \frac{1}{2} \left(\frac{\pi}{2}+\arcsin(y)\right) + \frac{1}{2} y \big(1-y^2\big)^{1/2}.
\end{gather*}
Therefore, $a_0^1= a_1^1 = 1/2$. Consequently, the recurrence implies that $a_0^\lambda = a_1^\lambda$ for all
$\lambda \in \N$.
\end{proof}

Analogously, define the indefinite integral
\begin{gather} \label{eq:defn_of_Ht}
H_\beta(y)= \int \big(1 -y^2\big)^{-(2\beta+1)/2} {\rm d}y, \qquad \beta > -1/2.
\end{gather}
This indefinite integral is often well defined when the corresponding definite integral $G_{-\beta} $ of equation~\eqref{eq:Glambda_def} is not.
\begin{Lemma}Some representatives of the indefinite integral $H_\beta$ are
\begin{gather}
H_0 (y) = \int \frac{1}{\sqrt{1-y^2}} {\rm d}y = \frac{\pi}{2} + \arcsin(y), \label{eq:H0}
\end{gather}
and
\begin{gather}
H_1(y) = \int \big(1 -y^2\big)^{-3/2} {\rm d}y = \frac{y}{\sqrt{1-y^2}}. \label{eq:H1}
\end{gather}
Families of representatives may be generated by the recurrence
\begin{gather}
H_\beta(y) = \frac{1}{2\beta-1}y \big(1-y^2\big)^{-(2\beta-1)/2} + \frac{2\beta-2}{2\beta-1} H_{\beta-1}(y), \qquad t \geq 1.
\label{eq:Ht_recur}
\end{gather}
Explicitly, for $\lambda \in \N$, starting from $H_1$, as given by equation~\eqref{eq:H0}, the recurrence generates representatives of the form
\begin{gather}
H_\lambda(y) = b^\lambda_1 y \big(1-y^2\big)^{-1/2}\! +b^\lambda_2 y \big(1-y^2\big)^{-3/2}\! + \cdots +
b^\lambda_\lambda y \big(1-y^2\big)^{-(2\lambda-1)/2}, \qquad\! \lambda \in \N.\!\! \label{eq:Hk_form}
\end{gather}
where
\begin{gather*} 
b_j^\lambda = \begin{cases} \dfrac{1}{2j-1} \dfrac{(2j-1)!!}{(2j-2)!!} d\dfrac{(2\lambda-2)!!}{(2\lambda-1)!!},
& 1 \leq j \leq \lambda,\\
0, & \text{otherwise}.
\end{cases}
\end{gather*}
\end{Lemma}
\begin{proof}Equations~\eqref{eq:H0} and \eqref{eq:H1} follow immediately from the definition~\eqref{eq:defn_of_Ht}. The
recurrence~\eqref{eq:Ht_recur} follows from the definition~\eqref{eq:defn_of_Ht} via an easy integration by parts. The general form of the explicit expression for $H_\lambda(y)$, $\lambda\in \N$, given in equation~\eqref{eq:Hk_form}, is clear from the expression for $H_1(y)$ and the recurrence~\eqref{eq:Ht_recur}. Consider now the expression for the coefficients $b_j^\lambda$ ocurring in equation~\eqref{eq:Hk_form}. From the recurrence and the formula for $H_1 $ the term involving $ y \big(1-y^2\big)^{-(2j-1)/2}$ first appears for $\lambda=j$, where it has the value $b_j^j =1/(2j-1)$. This term is then propagated to the functions $H_\lambda$, with $\lambda>j$, via the recurrence. Hence,
\begin{gather*}
b^\lambda_j = \frac{1}{2j-1} \frac{2j}{2j+1} \frac {2j+2}{2j+3} \cdots \frac{2\lambda-2}{2\lambda-1}
 = \frac{1}{2j-1} \frac{(2\lambda-2)!!}{(2j-2)!!} \frac {(2j-1)!!}{(2\lambda-1)!!}, \qquad 1 \leq j \leq \lambda.\tag*{\qed}
\end{gather*}\renewcommand{\qed}{}
\end{proof}

Given the formula \eqref{eq:Glambda_form} for $G_\lambda(y)$ when $\lambda$ is a positive integer, the definition~\eqref{eq:defn_of_Ht} of~$H_\lambda$ and the definition~\eqref{eq:Slambda} of~$S_\lambda$,
\begin{gather*}
S_\lambda(y) = \int \big(1 -y^2\big)^{-(2\lambda +1)/2} \bigg[ a_0^\lambda \left( \frac{\pi}{2} + \arcsin(y) \right) \\
\hphantom{S_\lambda(y) =}{} +a_1^\lambda y \big(1 -y^2\big)^{1/2} + \cdots +
a_\lambda^\lambda y \big(1-y^2\big)^{(2\lambda -1)/2} \bigg] {\rm d}y\\
\hphantom{S_\lambda(y)}{} = \int a_0^\lambda \left( \frac{\pi}{2} +\arcsin(y) \right) {\rm d} H_\lambda(y) \\
\hphantom{S_\lambda(y) =}{}
 +\int a_1^\lambda y \big(1 -y^2\big)^{-\lambda} + a_2^\lambda y \big(1 -y^2\big)^{-\lambda +1} + \cdots +
a_\lambda^\lambda y \big(1 -y^2\big)^{-1} {\rm d}y= I_1 + I_2,
\end{gather*}
where $I_1$ and $I_2$ are the first and second indefinite integrals, respectively. Ignoring the constant parts in the indefinite integrals a~representative value of~$I_2$ is
\begin{gather*}
I_2 = a_1^\lambda \frac{\big(1 -y^2\big)^{-(\lambda-1)}}{2(\lambda-1)} + a_2^\lambda \frac{\big(1-y^2\big)^{-(\lambda-2)} }{2(\lambda-2)}
+ \cdots + a_{\lambda-1}^\lambda\frac{\big(1-y^2\big)^{-1}}{2} - a_\lambda^\lambda \frac{\ln \big(1 -y^2\big)}{2} .
\end{gather*}
A representative value of $I_1$ is
\begin{gather*}
I_1 = a_0^\lambda \left[ \left( \frac{\pi}{2} + \arcsin(y) \right) H_\lambda(y) - \int H_\lambda (y) \frac{1}{\sqrt{1 -y^2}} {\rm d}y \right] \\
\hphantom{I_1}{} = a_0^\lambda \left( \frac{\pi}{2} + \arcsin(y)\right) H_\lambda (y)\\
\hphantom{I_1=}{}
- a_0^\lambda \int b_1^\lambda y \big(1 -y^2\big)^{-1} +b_2^\lambda y \big(1 -y^2\big)^{-2} + \cdots +b_\lambda^\lambda y \big(1 -y^2\big)^{-\lambda} {\rm d}y \\
\hphantom{I_1}{} = a_0^\lambda \left( \frac{\pi}{2} + \arcsin(y)\right) H_\lambda (y) \\
\hphantom{I_1=}{} +a_0^\lambda \left[ \frac{b_1^\lambda \ln \big(1 -y^2\big)}{2} + \frac{b_2^\lambda \big(1 -y^2\big)^{-1}}{2}
+ \frac{b_3^\lambda \big(1-y^2\big)^{-2} }{4} + \cdots + b_\lambda^\lambda \frac{\big(1-y^2\big)^{-(\lambda-1)}}{2\lambda -2} \right].
\end{gather*}

Now note that
\begin{gather*}
a_0^\lambda b_1^\lambda = \frac{(2\lambda-1)!!}{(2\lambda)!!} \frac{ (2\lambda -2)!!}{(2\lambda-1)!!} = \frac{1}{2\lambda} = a_\lambda^\lambda .
\end{gather*}
Hence, the terms in $I_1$ and $I_2$ involving $\ln \big(1 -y^2\big)$ have coefficients of equal magnitude and opposite sign. Thus, we conclude that a representative value of $S_\lambda (y)$ is
\begin{gather*}
S_\lambda(y) = \left(\frac{\pi}{2} + \arcsin(y) \right)\sum_{j=1}^\lambda c_j^\lambda y \big(1 -y^2\big)^{-j+\frac{1}{2}} + \sum_{j=1}^{\lambda-1} d_j^\lambda \big(1 -y^2\big)^{-j} ,
\end{gather*}
where
\begin{gather*} c_j^\lambda = a_0^\lambda b_j^\lambda = \frac{1}{2 \lambda } \frac{1}{2j-1} \frac{(2j-1)!!}{(2j-2)!!} ,
\end{gather*}
and
\begin{gather*} d_j^\lambda =\frac{1}{2j} \big( a_{\lambda-j}^\lambda - c_{j+1}^\lambda \big)\\
\hphantom{d_j^\lambda}{} =\frac{1}{2j} \left[ \frac{(2\lambda-1)!!}{ (2\lambda-2j-1)!!} \frac{ (2\lambda-2j-2)!!}{ ( 2\lambda )!!} - \frac{1}{2\lambda} \frac{(2j-1)!!}{(2j)!!} \right], \qquad 1 \leq j \leq \lambda-1.
\end{gather*}

Now recall from equation~\eqref{eq:alternative_formula_for_k_d1} that
\begin{gather*}
k_{d,1} = S_\lambda(y) -\frac{[S_\lambda, e_0]_\lambda}{[e_0,e_0]_\lambda},
\end{gather*}
where $\lambda=(d-2)/2$. To calculate this quantity first define $f_\mu$ as the Beta integral
\begin{gather*} 
f_\mu = \int_{-1}^1 \big(1 -y^2\big)^\mu {\rm d}y = \frac{\sqrt{\pi} \Gamma(\mu+1)}{\Gamma(\mu+3/2)},\qquad \mu > -1.
\end{gather*}
Then
\begin{gather*}
[S_\lambda,e_0]_\lambda = \int_{-1}^1 \sum_{j=1}^{\lambda-1} d_j^\lambda \big(1-y^2\big)^{-j} \big(1-y^2\big)^{\lambda -\frac{1}{2}} {\rm d}y \\
\hphantom{[S_\lambda,e_0]_\lambda =}{} + \int_{-1}^1 \sum_{j=1}^\lambda c_j^\lambda \left( \frac{\pi}{2} + \arcsin(y) \right) y
\big(1 -y^2\big)^{-j+\frac{1}{2}} \left(1 -y^2 \right)^{\lambda -\frac{1}{2}} {\rm d}y \\
\hphantom{[S_\lambda,e_0]_\lambda}{}= \sum_{j=1}^{\lambda-1} d_j^\lambda f_{\lambda -j -\frac{1}{2}}
+ \sum_{j=1}^\lambda c_j^\lambda
 \int_{-1}^1 \left( \frac{\pi}{2} + \arcsin(y) \right) {\rm d} \left( \frac{ - \big(1 -y^2\big) ^{\lambda -j +1}}{2 (\lambda -j +1)} \right) \\
\hphantom{[S_\lambda,e_0]_\lambda}{} = \sum_{j=1}^{\lambda-1} d_j^\lambda f_{\lambda -j -\frac{1}{2}} + \sum_{j=1}^\lambda
 c_j^\lambda \int_{-1}^1 \frac{ \big(1-y^2\big)^{\lambda-j+1}}{2(\lambda-j+1)} \frac{1}{\sqrt{1-y^2}} {\rm d}y \\
\hphantom{[S_\lambda,e_0]_\lambda}{}= \sum_{j=1}^\lambda
 \frac{1}{2(\lambda-j+1)} c_j^\lambda f_{\lambda-j +\frac{1}{2}} + \sum_{j=1}^{\lambda-1} d_j^\lambda f_{\lambda-j -\frac{1}{2}}.
 \end{gather*}
 Therefore, substituting the various quantities into equation~\eqref{eq:alternative_formula_for_k_d1}
 \begin{gather*} 
 k_{d,1} (y) = \left(\frac{\pi}{2} + \arcsin(x) \right)\sum_{j=1}^\lambda c_j^\lambda y \big(1 -y^2\big)^{-j+\frac{1}{2}}+ \sum_{j=1}^{\lambda-1} d_j^\lambda \big(1 -y^2\big)^{-j} -C_\lambda ,
 \end{gather*}
 where
 \begin{gather*}
 C_\lambda = \frac{1}{f_{\lambda-\frac{1}{2}}}\left(\sum_{j=1}^\lambda \frac{1}{2(\lambda-j+1)} c_j^\lambda f_{\lambda-j +\frac{1}{2}} + \sum_{j=1}^{\lambda-1} d_j^\lambda f_{\lambda-j -\frac{1}{2}}\right) .
 \end{gather*}

\subsection[The functions $k_{d,1}$ when $d>1$ is odd]{The functions $\boldsymbol{k_{d,1}}$ when $\boldsymbol{d>1}$ is odd} \label{subsection8_2}
Now turn to the calculation of the functions $k_{d,1}$ when $d>1$ is odd. Then the Gegenbauer parameter $\lambda = (d-2)/2 = \kappa +\half$, for some nonnegative integer $\kappa$.

We will need the following technical lemmas.
\begin{Lemma} \label{lem:technical1}Let $\alpha$ be an integer and $\beta$ a nonnegative integer. If $\alpha$ is nonnegative further suppose $\beta < a$. Then
\begin{gather*}
J_{\beta,\alpha} =\sum_{\nu=0}^\beta \binom{\beta}{\nu} \frac{(-1)^\nu}{\alpha-\nu}=\frac{(-1)^\beta \beta!}{\alpha(\alpha-1) \cdots (\alpha-\beta)}.
\end{gather*}
\end{Lemma}
\begin{proof}
The identity will be proven by induction on $\beta$. For $\beta=0$ the result is immediate. Now assume that the identity is true for $\beta=\kappa-1$ where $\kappa \in \N$. Then
\begin{gather*}
J_{\kappa,a} = \binom{\kappa}{0}\frac{1}{\alpha} +\binom{\kappa}{\kappa} \frac{(-1)^\kappa}{\alpha-\kappa}+ \sum_{\nu=1}^{\kappa-1}\left\{ \binom{\kappa-1}{\nu}+\binom{\kappa-1}{\nu-1}\right\} \frac{(-1)^\nu}{\alpha-\nu}\\
\hphantom{J_{\kappa,a}}{} = \sum_{\nu=0}^{\kappa-1} \binom{\kappa-1}{\nu} \frac{(-1)^\nu}{\alpha-\nu} + \left\{ \frac{(-1)^\kappa}{\alpha-\kappa}+ \sum_{\ell=0}^{\kappa-2} \binom{\kappa-1}{\ell}
	 \frac{(-1)^{\ell+1}}{\alpha-1-\ell} \right\}\\
\hphantom{J_{\kappa,a}}{} = \sum_{\nu=0}^{\kappa-1} \binom{\kappa-1}{\nu} \frac{(-1)^\nu}{\alpha-\nu} - \sum_{\ell=0}^{\kappa-1} \binom{\kappa-1}{\ell} \frac{ (-1)^\ell }{\alpha-1-\ell} .
\end{gather*}
Applying the induction hypothesis twice
\begin{gather*}
J_{\kappa,\alpha} = \frac{(-1)^{\kappa-1} (\kappa-1)!}{\alpha(\alpha-1) \cdots (\alpha-\kappa+1)}
 - \frac{ (-1)^{\kappa-1} (\kappa-1)!}{(\alpha-1)(\alpha-2)\cdots (\alpha-\kappa )}=\frac{(-1)^\kappa \kappa! }{\alpha (\alpha-1) \cdots (\alpha-\kappa)},
\end{gather*}
showing that the identity also holds for $\beta=\kappa$.
\end{proof}

\begin{Lemma} \label{lem:technical2}Let $\kappa\in \N_0$. Then
\begin{gather*}
\int_{-1}^1 \left(1-x^2\right)^\kappa \ln \left( \frac{1-x}{2} \right) {\rm d}x = 2^{2\kappa+1}(-1)^{\kappa+1} \sum_{\nu=0}^\kappa \binom{\kappa}{\nu} \frac{(-1)^\nu}{(2\kappa-\nu+1)^2}.
\end{gather*}
\end{Lemma}
\begin{proof}
\begin{gather}
I = \int_{-1}^1 \left(1-x^2\right)^\kappa \ln \left( \frac{1-x}{2} \right) {\rm d}x = \sum_{\nu=0}^\kappa \binom{\kappa}{\nu} 2^\nu (-1)^\nu \int_{-1}^1 \ln\left( \frac{1-x}{2} \right) (1-x)^{2\kappa-\nu} {\rm d}x
 \nonumber \\
\hphantom{I} =\sum_{\nu=0}^\kappa \binom{\kappa}{\nu} 2^{2\kappa} (-1)^\nu \int_{-1}^1 \ln\left( \frac{1-x}{2} \right) \left( \frac{1-x}{2} \right)^{2\kappa-\nu} {\rm d}x. \label{eq:tech_lem2_I}
\end{gather}
Now
\begin{gather*}
\int_{-1}^1 \ln \left( \frac{1-x}{2} \right) \left( \frac{1-x}{2} \right)^{2\kappa-\nu} {\rm d}x = 2 \int_{0}^1 \ln( t) t^{2\kappa-\nu} {\rm d}t= \frac{ -2}{(2\kappa-\nu+1)^2}.
\end{gather*}
Substituting into equation~\eqref{eq:tech_lem2_I} yields the result.
\end{proof}

We now turn to the development of an expression for $G_\lambda$ which will be particularly convenient for the evaluation of the indefinite integral $S_\lambda (y)$ of equation~\eqref{eq:Slambda} in this $\lambda=\kappa+\half$ case. Recall the definition \eqref{eq:Glambda_def} of $G_\lambda^\alpha$. In this section we restrict ourselves to the case where $\alpha$ is a nonnegative integer, with $\alpha \leq \kappa$.
Substituting $1-z=2-(1+z)$ into the expression for $G_\lambda^\alpha$ yields
\begin{gather}
 G_\lambda^\alpha (y) = \int_{-1}^y \big(1-z^2\big)^\kappa (1-z)^{-\alpha} {\rm d}z \nonumber\\
\hphantom{G_\lambda^\alpha (y)}{} = \int_{-1}^y (1+z)^\kappa \sum_{\gamma=0}^{\kappa-\alpha} \binom{\kappa-\alpha}{\gamma} 2^\gamma (-1)^{\kappa-\alpha-\gamma} (1+z)^{\kappa-\alpha-\gamma} {\rm d}z \nonumber\\
\hphantom{G_\lambda^\alpha (y)}{}= \sum_{\gamma=0}^{\kappa-\alpha} \binom{\kappa-\alpha}{\gamma} \frac{2^\gamma (-1)^{\kappa-\alpha-\gamma}}{2\kappa-\alpha -\gamma+1} (1+y)^{2\kappa-\alpha-\gamma+1}.
\label{special_form_of_G_lambda_alpha}
\end{gather}
In particular, for $\alpha=0$,
\begin{gather*}
G_\lambda (y) = (1+y)^{\kappa+1} \widetilde{G}_\kappa (y),
\end{gather*}
where
\begin{gather*}
\widetilde{G}_\lambda (y) = \sum_{\gamma=0}^\kappa \binom{\kappa}{\gamma} \frac{2^\gamma (-1)^{\kappa-\gamma}}{2\kappa-\gamma+1} (1+y)^{\kappa-\gamma}\\
\hphantom{\widetilde{G}_\lambda (y)}{} = \sum_{\gamma=0}^\kappa \binom{\kappa}{\gamma} \frac{2^\gamma (-1)^{\kappa-\gamma}}{2\kappa-\gamma+1}
 \sum_{\ell=0}^{\kappa-\gamma} \binom{\kappa-\gamma}{\ell} 2^\ell (-1)^{\kappa-\gamma-\ell} (1-y)^{\kappa-\gamma-\ell}.
\end{gather*}
Now substituting $\nu=\gamma+\ell$ and noting that $ \binom{\kappa}{\gamma} \binom{\kappa-\gamma}{\nu-\gamma}=\binom{\kappa}{\nu}\binom{\nu}{\gamma}$
\begin{gather*}
\widetilde{G}_\lambda (y) = \sum_{\gamma=0}^\kappa \sum_{\nu=\gamma}^\kappa 	\binom{\kappa}{\gamma}\binom{\kappa-\gamma}{\nu-\gamma} \frac{2^\nu (-1)^{\nu-\gamma} }{2\kappa-\gamma+1} (1-y)^{\kappa-\nu} \\
\hphantom{\widetilde{G}_\lambda (y)}{} = \sum_{\nu=0}^\kappa 2^\nu \binom{\kappa}{\nu} \left\{
 \sum_{\gamma=0}^\nu \binom{\nu}{\gamma} (-1)^{\nu-\gamma} \frac{1}{2\kappa-\gamma+1} \right\} (1-y)^{\kappa-\nu}.
\end{gather*}
Applying Lemma~\ref{lem:technical1} finally gives
\begin{gather*}
\widetilde{G}_\lambda (y)= \sum_{\nu=0}^\kappa 2^\nu \binom{\kappa}{\nu} \frac{ \nu! (2\kappa-\nu)!}{(2\kappa+1)!} (1-y)^{\kappa-\nu}.
\end{gather*}
Therefore,
\begin{gather*} G_\lambda (y) = (1+y)^{\kappa+1} \sum_{\nu=0}^\kappa 2^\nu \binom{\kappa}{\nu}\frac{ \nu! (2\kappa-\nu)!}{(2\kappa+1)!} (1-y)^{\kappa-\nu},
\end{gather*}
and hence a representative value of $S_\lambda$ is
\begin{gather*}
S_\lambda(y) = \int \big(1-y^2\big)^{-\kappa-1} G_\lambda (y) {\rm d}y = \sum_{\nu=0}^\kappa 2^\nu \binom{\kappa}{\nu}
\frac{ \nu! (2\kappa-\nu)!}{(2\kappa+1)!}\int (1-y)^{-1-\nu } {\rm d}y \\
\hphantom{S_\lambda(y)}{}= g^\lambda_0 \ln \left(\frac{1-y}{2}\right) + \sum_{\nu=1}^\kappa g^\lambda_\nu(1-y)^{-\nu},
 \end{gather*}
 where
 \begin{gather}\label{eq:def_g_lambda_nu}
 g^\lambda_\nu = \begin{cases} \dfrac{-1}{2\kappa+1}, &\text{when}\ \nu=0,\\
 \dfrac{1}{2\kappa+1},& \text{when}\ \nu=1,\\
 \displaystyle 2^\nu \binom{\kappa}{\nu}
\frac{ (\nu-1)! (2\kappa-\nu)!}{(2\kappa+1)!},& \text{when}\ 1<\nu\leq \kappa.
\end{cases}
\end{gather}
To complete the calculation of $k_{d,1}$ we need to evaluate the constant
\begin{gather*}
[ S_\lambda , e_0]_\lambda = \int_{-1}^1 \big(1-x^2\big)^\kappa g^\lambda_0 \ln\left(\frac{1-x}{2}\right) {\rm d}y
+ \int_{-1}^1 \big(1-x^2\big)^\kappa \sum_{\nu=1}^\kappa g^\lambda_\nu(1-x)^{-\nu} {\rm d}x = I_1 +I_2.
\end{gather*}
The integral $I_1$ is given in Lemma~\ref{lem:technical2}. To compute $I_2$ we apply equation~\eqref{special_form_of_G_lambda_alpha} which implies that for $0 \leq \nu \leq \kappa$
\begin{gather*}
 \int_{-1}^1 \big(1-x^2\big)^\kappa (1-x)^{-\nu} {\rm d}x =G^\nu_\lambda(1) = \sum_{\gamma=0}^{\kappa-\nu} \frac{\binom{\kappa-\nu}{\gamma} 2^\gamma (-1)^{\kappa-\nu-\gamma} 2^{2\kappa-\nu-\gamma+1}}{2\kappa-\nu-\gamma+1} \\
\hphantom{\int_{-1}^1 \big(1-x^2\big)^\kappa (1-x)^{-\nu} {\rm d}x}{} = 2^{2\kappa-\nu+1} (-1)^{\kappa-\nu} \sum_{\gamma=0}^{\kappa-\nu} \frac{ \binom{\kappa-\nu}{\gamma} (-1)^\gamma}{2\kappa-\nu -\gamma+1} .
 \end{gather*}
 An application of Lemma~\ref{lem:technical1} then shows
 \begin{gather*}
G^\nu_\lambda(1) = \frac{2^{2\kappa-\nu +1} \kappa! (\kappa-\nu)!}{(2\kappa-\nu +1)!}.
 \end{gather*}
 Putting all these things together, for $d=2\kappa+1 >1$ odd, and $\lambda=(d-2)/2$,
 \begin{gather*} 
 k_{2\kappa+1,1}(x) =\frac{-1}{2\kappa+1} \ln\left(\frac{1-x}{2}\right) + \sum_{\nu=1}^\kappa g^\lambda_\nu(1-x)^{-\nu} -D_\lambda,
 \end{gather*}
 where
 \begin{gather*}
 D_\lambda =\frac{ [S_\lambda,e_0]_\lambda}{[e_0,e_0]_\lambda} = \frac{1}{f_{\lambda-\half}} \left(\sum_{\nu=1}^{\kappa} g^\lambda_\nu
G_\lambda^\nu(1) +\frac{1}{2\kappa+1} 2^{2\kappa+1} (-1)^\kappa \sum_{\nu=0}^\kappa \frac{\binom{\kappa}{\nu} (-1)^\nu}{(2\kappa-\nu+1)^2} \right).
 \end{gather*}

\pdfbookmark[1]{References}{ref}
\LastPageEnding

\end{document}